\documentclass{amsart}
\usepackage{amssymb,amsmath,amsfonts,pictex,graphicx,fullpage,etex,tikz,hyperref}
\usetikzlibrary{shapes,shapes.arrows,decorations,positioning,matrix,snakes,decorations.markings,arrows}
\usepackage{soul}

\newcommand{\eS}{\Sigma}

\newcommand{\R}{\mathbb{R}}
\newcommand{\E}{\mathbb{E}}
\newcommand{\N}{\mathbb{N}}

\newcommand{\Max}{{\sf Max}}

\newcommand{\RalfSubsection}[1]{\subsection{\sc #1}}

\theoremstyle{plain}
\newtheorem{prop}{Proposition}[section]

\newtheorem{lemma}[prop]{Lemma}
\newtheorem{corollary}[prop]{Corollary}

\newtheorem{Thm}[prop]{Theorem}

\theoremstyle{definition}

\newtheorem{rem}[prop]{Remark}

\tikzstyle{point}=[circle,fill=black,scale=.4]
\pgfdeclarelayer{edgelayer}
\pgfdeclarelayer{nodelayer}
\pgfsetlayers{edgelayer,nodelayer,main}

\begin{document}

\title
	[Affine Maps Between CAT(0) Spaces]
	{Affine Maps Between CAT(0) Spaces}

\author{Hanna Bennett$^\dagger,$ Christopher Mooney$^\ddagger$, and Ralf Spatzier$^\ast$}

\thanks{$^{ \dagger}$ Supported in part by an EAF grant for the  IBL Center at the University of Michigan}

\thanks{$^{\ddagger}$ Supported in part by the NSF EMSW21 grant RTG-0602191 and  a Caterpillar Fellowship}

\thanks{$^{\ast }$ Supported in part by the NSF grants DMS-0906085 and DMS-1307164}

\begin{abstract}
We study affine maps between CAT(0) spaces with geometric actions, and show that they essentially split
as products of  dilations and linear maps (on the Euclidean factor).
This extends known results from the Riemannian case.
Furthermore, we prove a splitting lemma for the Tits boundary of a CAT(0) space with geometric action,
a variant of a splitting lemma for geodesically complete CAT(1) spaces by Lytchak.
\end{abstract}

\maketitle

\section{Introduction}
Let $X$ and $Y$ be geodesic spaces, and $f\colon X \mapsto Y$ a map.
Recall that a \textit{geodesic space} is a metric space in which every pair of points is joined by a path
of shortest length, called a \textit{geodesic}.  
We will always parametrize geodesics by arc length.
We call $f$ {\em affine} if $f$ maps geodesics $\gamma$ in $X$ to geodesics in $Y$,
and $f$ rescales $\gamma$ with constant speed $\rho(\gamma)$, which a priori depends on the geodesic $\gamma$. We call $\rho$ the \textit{rescaling function}, and say that an affine map is a {\em dilation} if the rescaling function is constant. In this paper we classify affine maps between CAT(0) spaces.
We will not touch on the much more difficult question of the extent to which the set of geodesics determines the metric.
Matveev has obtained strong positive results, especially for closed Riemannian manifolds with negative curvature \cite{Matveev2003}.

If $X$ and $Y$ are Riemannian, the answer to our problem  is classical  \cite{Kobayashi1963I}:  any self-affine map of an irreducible, non-Euclidean Riemannian manifold is a dilation. Remarkably, Lytchak  \cite{Lytchak2012},
following work by Ohta \cite{Ohta2003}, classified affine maps from a Riemannian manifold $X$ to any metric space $Y$ as dilations, as long
as $X$ is not a product or a higher rank symmetric space.  In the latter cases, they produce counterexamples by endowing
these spaces with suitable Finsler metrics.

Lytchak and Schr\"{o}der \cite{Lytchak-Schroeder}, and later Hitzelberger and Lytchak \cite{Hitzelberger2007} further investigated the case of real-valued
functions on a CAT($\kappa$) metric space, and obtained severe restrictions.  
Understanding affine maps has been important in several applications, and has connections to superrigidity problems.  We refer the reader to Ohta \cite{Ohta2003} for a brief discussion.  

To state our main result, we recall a fundamental result about splittings of metric spaces. Foertsch and Lytchak \cite[Theorem 1.1]{Foertsch2008} prove that such splittings exist and are essentially unique when $X$ is a geodesic metric space of finite topological dimension (or, more generally, finite affine rank), i.e. under those conditions $X=X_1\times ...\times X_n\times\E^d$ such that each $X_i\neq\R$ and is irreducible. In particular their result applies to CAT(0) spaces with a geometric action, that is, the action is properly discontinuous by isometries with compact
quotient. This case was proved earlier by Caprace and Monod \cite{CapraceMonod2009}.  

Recall that a metric space is called \textit{proper} if all closed balls are compact.

\begin{Thm}[Main Theorem]
\label{main}
Let $X$ be a geodesically complete proper CAT(0) space admitting a geometric group action, and $X=X_1\times ...\times X_n\times\E^d$
be a factorization into irreducible factors such that no $X_i=\R$.  Let $Y$ also be a CAT(0) space,
and $f\colon X\to Y$ be a continuous affine map.  Then the image of $f$ is convex and hence CAT(0) and splits as $f(X_1)\times ... \times f(X_n)\times f(\E^d)$, where $f(\E^d)$ is the Euclidean factor of the image.  The restriction
of $f$ to every factor $X_i$ is a dilation (possibly with rescaling
constant zero) and the restriction to $\E^d$ is a standard affine map between Euclidean spaces. 
\end{Thm}

Combining the Main Theorem with work of Bosch\'{e} \cite{Bosche}, we get two applications
to self maps of CAT(0) spaces.

\begin{corollary}[Self-Affine Maps]
\label{thm-selfmaps}
Let $X$ be a proper geodesically complete CAT(0) space with geometric action and $f\colon X\to X$ be a strictly contracting continuous affine map.
Then $X$ is flat.
\end{corollary}

\begin{corollary}[Self-Affine Homeomorphisms]
\label{thm-selfhomeo}
Let $X$ be a proper geodesically complete CAT(0) space admitting a geometric group action and $f:X\to X$ be an
affine homeomorphism.  Assume further that $X$ has no Euclidean factor.
\begin{enumerate}
\item If $f$ preserves the factors of $X$ then $f$ is an isometry.
\item If $X$ is irreducible, then $f$ is an isometry.
\item Some power of $f$ is an isometry.
\item The group of isometries has finite index in the group of affine homeomorphisms.
\end{enumerate}
\end{corollary}

One of the key tools in the proof of the Main Theorem is the following Splitting Lemma, which is
an analogue to Lytchak's result \cite[Proposition~4.2]{Lytchak2005} for splittings of geodesically complete CAT(1) spaces.
We will work with boundaries of CAT(0) spaces admitting geometric actions.  Endowed with the Tits metric,
such boundaries are always CAT(1) spaces but often not geodesically complete.  See, for instance,
the Croke-Kleiner examples \cite{Croke2002}.
A subset $P$ of a CAT(1) space $Y$ is called \textit{$\pi$-convex} if whenever $x,y\in P$ such that $d(x,y)<\pi$,
then the unique geodesic joining them is also in $P$.  Given $x,y\in Y$, $x$ and $y$ are called \textit{antipodal}
or \textit{antipodes} if $d(x,y)\ge\pi$.
A nonempty subset $P\subset Y$ is \textit{involutive} if it contains all of its antipodes.

\begin{Thm}[Splitting Lemma]
\label{spatziers lemma}
Let $X$ be a proper CAT(0) space admitting a geometric group action.  Suppose
$\partial X$ contains a proper subset $P$ that is $\pi$-convex and involutive and closed in the
cone topology.  Then $\partial X$ splits as the spherical join $\partial X=P\ast P^\perp$
where $P^\perp$ is the set of points that have Tits distance exactly $\pi/2$ from all points in $P$.
\end{Thm}

We note that our definition of the perpendicular set $P^\perp$ is different from Lytchak's $\textrm{Pol}(P)$, which
is the set of points that are at least $\pi/2$ from all points in $P$.

To prove the Splitting Lemma, we critically use the $\pi$-convergence theorem of
Papasoglu-Swenson \cite{Papasoglu2009} and the theorem of Kleiner \cite{Kleiner1999} that
the boundary $\partial X$ of $X$ contains isometrically embedded round
spheres of dimension equal to the geometric dimension
of the Tits boundary (cf. \S\ref{subsec:tools}).

We then apply the Splitting Lemma to prove the Main Theorem.
We first show that asymptotic geodesics are rescaled by the same constant.
Thus the rescaling function extends to the boundary $\partial X$.
If the rescaling function $\rho$ is not constant on $\partial X$, then let $P$ be the set of points
on which $\rho$ attains its maximum.
We show that $P$ is closed, $\pi$-convex, and involutive.
Therefore the Tits boundary splits off a factor.
We can then apply a theorem of Bridson-Haefliger \cite[Theorem~II.9.24]{Bridson1999}
to get that the underlying CAT(0) space splits as a product.

We do not know if our results extend to affine maps between CAT($\kappa$) spaces, or at least from CAT(0) to CAT($\kappa$) spaces.  
As we mentioned above, affine maps between irreducible metric spaces are not always dilations.
We also do not know if affine maps are always continuous.  This is certainly the case in many situations.

\subsection*{Acknowledgements}
We thank Dan Guralnik and Eric Swenson for their inspiring paper \cite{Guralnik2011}, and Russell Ricks and Ben Schmidt
for discussions of this paper.  The first two authors held postdoctoral fellowships at the University of Michigan, and are thankful for the excellent working environment provided.

\section{CAT(0) Spaces}

\RalfSubsection{Tools From CAT(0) Spaces} 
\label{subsec:tools}

In this section we review some definitions  and techniques in CAT(0) spaces.
For a more thorough treatment the reader should review \cite{Bridson1999}.
Let $X$ be a geodesic metric space. Given three distinct points $x,y,z\in X$, choose points $\overline x$, $\overline y$,
and $\overline z$ in the Euclidean plane $\E^2$ such that $d(\overline x,\overline y)=d(x,y)$,
$d(\overline x,\overline z)=d(x,z)$, and $d(\overline y,\overline z)=d(y,z)$.
We denote the resulting triangle in $\E^2$ by $\overline{\triangle}xyz$ and call it
a \textit{comparison triangle} for $\triangle xyz$ in $X$.
Choose any $p$ in the geodesic $[x,y]$ and $q$ in $[x,z]$ and
get corresponding points
$\overline p\in[\overline x,\overline y]$ and
$\overline q\in[\overline x,\overline z]$.
If for every choice of $p$ and $q$, we have 
$d(p,q)\le d(\overline p,\overline q)$, then $\triangle xyz$ is said to be
\textit{no fatter} than $\overline\triangle xyz$.  If every triangle in $X$
is no fatter than its comparison triangle in $\E^2$, then $X$ is called
\textit{CAT(0)}.

We will henceforth assume that our CAT(0) spaces are {\em geodesically complete} (i.e., that every geodesic segment extends to a geodesic line defined on $(-\infty, \infty)$) and that $X$ is proper (i.e., that closed balls are
compact).

Next we recall the definition of Alexandrov angles in CAT(0) spaces. If $x,y,z\in X$, then we denote by $\overline{\angle}_{x}(y,z)$ the corresponding
angle in the comparison triangle $\overline\triangle xyz$.  If
$\alpha$ and $\beta$ are the geodesics $[x,y]$ and $[x,z]$, then the
CAT(0) condition implies that
\begin{align}
\label{angle function}
	t\mapsto \overline\angle_x\bigl(\alpha(t),\beta(t)\bigr) \;\textrm{is a nondecreasing function}.
\end{align}
Its limit as $t\to 0$ is called the
\textit{Alexandrov angle} between $\alpha$ and $\beta$, denoted by
$\angle_x(y,z)$, or $\angle_x(\alpha,\beta)$.  The condition that
$\angle_x(y,z)\le\overline{\angle}_x(y,z)$ for every choice of $x,y,z \in X$
is equivalent to the CAT(0) property for $X$
\cite[Proposition~II.1.7(4)]{Bridson1999}.

Many geometric properties of nonpositively curved manifolds carry over to the
CAT(0) setting.  For instance, angle sums of triangles are bounded above by $\pi$. Furthermore we can define the  \textit{visual boundary} of $X$, denoted $\partial X$, as the set of equivalence
classes of geodesic rays in $X$. Two geodesic rays $\alpha$ and $\beta$
are equivalent, or asymptotic, if one lies in a tubular
neighborhood of the other.  Equivalently, if a basepoint $x_0\in X$ is fixed,
then $\partial X$ may be defined as the set of geodesic rays emanating from $x_0$ \cite[\S II.8]{Bridson1999}.
We think of $\partial X$  as ``attached to $X$ at infinity'' and it captures the
notion of infinity of $X$.

We endow $\overline{X}=X\cup\partial X$ with a topology by
identifying points in $X$ with geodesics emanating from a common basepoint $x_0$
and points in $\partial X$ with geodesic rays emanating from the same point.
Then a sequence of points $(x_n)\subset\overline{X}$ converges to a point
$y\in\overline{X}$ if the corresponding geodesics converge uniformly on compact sets.
The subspace topology on $\partial X$ is called the \textit{cone topology}.
When $X$ is proper, $\overline{X}$ is a compactification for $X$ and $\partial X$ is compact.
In this topology geodesic rays are close if they track a long time before diverging.

A second topology on $\partial X$ comes from a metric.
Given a pair of points
$\zeta,\eta\in\partial X$, the \textit{Tits angle} between them,
denoted $\angle_{Tits}(\zeta,\eta)$,  is defined as the supremum of
Alexandrov angles $\angle_x(\zeta,\eta)$ between the geodesic rays
$\alpha$ and $\beta$ emanating from $x$ going out to $\zeta$ and $\eta$ as $x$
ranges over $X$.  If $x$ is fixed, we know from \cite[Proposition~II.9.8(1)]{Bridson1999} that
\begin{align}
\label{titsangles}
	\lim_{t\to\infty}\overline{\angle}_x\bigl(\alpha(t),\beta(t)\bigr)
		=
	\angle_{Tits}(\zeta,\eta)
\end{align}
where $\overline{\angle}_x\bigl(\alpha(t),\beta(t)\bigr)$ is nondecreasing
by (\ref{angle function}).

The \textit{Tits metric} $d_{Tits}$ is the corresponding length metric (possibly taking the value infinity).
This metric induces the \textit{Tits topology}, which is finer than the cone topology.
The boundary with the Tits topology is called the \textit{Tits boundary} and denote it by $\partial_{Tits}X$.
It is well-known that $d_{Tits}$ is lower semicontinuous in the cone topology
\cite[Proposition~II.9.5(2)]{Bridson1999}.  Also, when $\angle_{Tits}(\zeta,\eta)<\pi$,
then $d_{Tits}(\zeta,\eta)=\angle_{Tits}(\zeta,\eta)$ \cite[Remark~II.9.19(2)]{Bridson1999}
and $d_{Tits}(\zeta,\eta)\ge\pi$ iff $\angle_{Tits}(\zeta,\eta)=\pi$.

Amazingly, the Tits boundary has an elegant geometry.  Given any real
$\kappa$, a geodesic space is called CAT($\kappa$) if triangles are
no fatter than comparison triangles in a simply connected Riemannian manifold
with constant curvature $\kappa$.
The following theorem is due to Gromov in the manifold setting \cite{Ballmann1985}, and to Bridson
and Haefliger in full generality.

\begin{Thm}\cite[Theorem~II.9.13]{Bridson1999}
Let $X$ be a complete CAT(0) space.  Then its Tits boundary is a complete CAT(1) space.
In particular, any two points of finite Tits distance are connected by a Tits geodesic.
\end{Thm}

Geodesics in $\partial_{Tits}X$ reflect flatness in $X$.
For instance, if $x\in X$ and $\zeta,\eta\in\partial X$ such that
$\angle_x(\zeta,\eta)=\angle_{Tits}(\zeta,\eta)$, then the convex hull of
the union of the two rays emanating from $x$ going out to $\zeta$ and $\eta$
is isometric to a sector in $\E^2$ \cite[Corollary~II.9.9]{Bridson1999}.

Another important example of a CAT(1) space is the
space of directions based at a point.
Given a point $y$ in a CAT($\kappa$) space $Y$, deem two geodesic segments
emanating from $y$ to be equivalent if the angle between them is zero.
Equivalence classes are called \textit{geodesic germs.}  The completion
of this space with the metric induced by angles is called the
\textit{link of $y$}.  Elements in the link are thought of as directions.
A direction is \textit{genuine} if it has a geodesic representative.
Links in a CAT(1) space are always CAT(1) \cite{Nikolaev}.

Following Kleiner in \cite{Kleiner1999}, we define the
\textit{geometric dimension} of a CAT($\kappa$) space $Y$ to be the
smallest function $\dim$ on the class of such spaces
(taking on non-negative integer values and infinity) such that
\begin{itemize}
\item $\dim Y=0$ if $Y$ is discrete and 
\item $\dim Y$ is strictly greater than the dimension of every link in $Y$ 
\end{itemize}
This is equal to the maximal topological dimension of compact subspaces of $Y$
\cite[Theorem~A]{Kleiner1999}. Moreover, if $Y$ is a CAT(0) space with geometric action, then the topological and geometric dimensions are equal and finite \cite[Theorem~C]{Kleiner1999}.

\RalfSubsection{Density of Round Spheres}
\label{subsec:density}
A \textit{round sphere} in a CAT(1) space $Y$ of geometric dimension $d>0$ is an isometrically
embedded $d$-sphere $\Sigma$ of curvature 1. In a zero-dimensional CAT(1) space, a \textit{round sphere} is just
a pair of points with distance $\infty$.  For boundaries of certain CAT(0) spaces, Kleiner
proved the existence of round spheres.

\begin{Thm}\cite[Theorem~C]{Kleiner1999}
Let $X$ be a CAT(0) space admitting a geometric group action. Let $d$ be the geometric dimension of its Tits boundary.
Then there is a round sphere of dimension $d$ in $\partial_{Tits}X$. In addition, this round sphere is the boundary of an isometrically embedded $(d+1)$-flat in $X$.
\end{Thm}

We remark that Leeb \cite[Proposition~2.1]{Leeb2000} proved that any top-dimensional round sphere in the boundary of $X$ is the boundary of a flat. 

We will also need the $\pi$-convergence technique of Papasoglu-Swenson:

\begin{Thm}\cite[Theorem~4]{Papasoglu2009}
Let $G$ be a group acting geometrically on a CAT(0) space $X$ and $\theta\in[0,\pi]$.
Then for any sequence of distinct elements $(g_n)\subset G$, there are points $p,q\in\partial X$
and a subsequence $(g_{n_k})$ such that for every compact $K\subset\partial X\setminus \overline{B_{Tits}}(p,\theta)$
and neighborhood $U$ of $\overline{B_{Tits}}(q,\pi-\theta)$, $g_{n_k}(K)\subset U$ for large enough $k$.
\end{Thm}

It will be important below to understand how $p$, $n$ and $(g_{n_k})$ arise.  Fix
$x_0\in X$.  Since $\partial X$ is compact,
we can pass to a subsequence so that the sequences $(g_{n_k}x_0)$ and $(g_{n_k}^{-1}x_0)$ converge to points 
$p$ and $n$ respectively.

The existence of antipodes was proven by Balser-Lytchak.

\begin{lemma}\cite[Lemma~3.1]{Balser2005}
\label{antipodes are}
Let $X$ be a CAT(1) space with geometric dimension $d<\infty$ and $K\subset X$ be a round sphere.
Then every point of $X$ has an antipode in $K$.
\end{lemma}

\begin{corollary}[Density of Round Spheres]
\label{density of round spheres}
Let $G$ be a group acting geometrically on a CAT(0) space $X$.
Then the union of round spheres in $\partial X$ is dense (in the the cone topology). In addition, we may pick these round spheres to consist of boundaries of flats in $X$.
\end{corollary}

\begin{proof}
Fix $x_0\in X$.
Choose any point $p\in\partial X$.  Since the action of $G$ is cocompact, there is a sequence of group elements
$(g_n)\subset G$ such that $(g_nx_0)$ converges to $p$.  After passing to a subsequence (if necessary), we may assume
that $(g_n^{-1}x_0)$ also converges to some point $n\in\partial X$.  By \cite[Theorem~C]{Kleiner1999} we know that there exists a round
sphere $K\subset\partial_{Tits}X$.  By Lemma~\ref{antipodes are}, we can get a $q\in K$
for which $d_{Tits}(n,q)\ge\pi$.
Apply $\pi$-convergence now to get that $g_nq\to p$.
Since $G$-translates of round spheres are round spheres, we see that the family of round spheres
$GK$ is dense.
\end{proof}

Suppose $X$ is a CAT(0) space with geometric group action. To prove the Main Theorem, we will need to know that any factor of $X$ also has a dense set of round spheres at infinity. We conclude this section by proving this.  Recall that a flat in a CAT(0) space $X$
is called \textit{maximal} if it is not contained as a proper subspace of any other flat in $X$.

\begin{lemma}
Let $X$ be a CAT(0) space splitting as $X=X_1\times X_2$.
Then $F\subset X$ is a maximal flat iff it can be written as $F=F_1\times F_2$
where $F_1$ and $F_2$ are maximal flats in $X_1$ and $X_2$ respectively.
\end{lemma}

\begin{proof}
Let $\pi_i:X\to X_i$ denote coordinate projection.  Let $F$ be a maximal flat in $X$ and denote $F_i=\pi_i(F)$.
Then $F_1=\pi_1(F)$ is a flat in $X_1$ because it is isometric to a totally geodesic subspace of $F$, namely
$F_1\times\{x_2\}$ where $x_2\in F_2$.  Similarly, $F_2$ is a flat in $X_2$.  Since $F_1\times F_2$ is a flat in $X$
containing $F$ and $F$ is maximal, we must have $F=F_1\times F_2$.  Finally, if $F_1$ were not maximal in $X_1$,
then there would be a flat $F_1'\subset X_1$ containing $F_1$ as a proper subspace and $F_1'\times F_2$ would contain
$F$ as a proper subspace.  So $F_1$ and $F_2$ must both be maximal.  This proves the forward implication.

Now suppose we have been given maximal flats $F_1,F_2$ in $X_1,X_2$, and suppose $F\subset X$ is a flat containing $F_1\times F_2$.
Then $\pi_1(F)$ is a flat in $X_1$ containing $F_1$ and $\pi_2(F)$ is a flat in $X_2$ containing $F_2$.  Therefore
$\pi_1(F)=F_1$ and $\pi_2(F)=F_2$, which means that $F\subset F_1\times F_2$.  Therefore $F=F_1\times F_2$, showing that
$F_1\times F_2$ is maximal.
\end{proof}

As an immediate consequence, we get

\begin{corollary}
Let $X$ be a CAT(0) space splitting as $X=X_1\times X_2$.
A subspace of $\partial_{Tits}X$ is a round sphere iff it is a spherical join of
round spheres in $\partial_{Tits}X_1$ and $\partial_{Tits}X_2$.
\end{corollary}

\begin{lemma}
\label{density in factors}
Let $X$ be a CAT(0) space splitting as $X=X_1\times X_2$ such that $\partial X$ has a dense family of round spheres.
Then the factors $\partial X_1$ and $\partial X_2$ also have dense families of round spheres.
\end{lemma}

\begin{proof}
Let $q:\partial X_1\times\partial X_2\times[0,\pi/2]\to\partial X_1\ast\partial X_2$ be a
quotient map with the conventions
\begin{itemize}
\item The restriction to $\partial X_1\times\partial X_2\times\{0\}$ is projection to the $\partial X_1$-coordinate. \\
\item The restriction to $\partial X_1\times\partial X_2\times\{\pi/2\}$ is projection to the $\partial X_2$-coordinate. \\
\item The restriction to $\partial X_1\cup\partial X_2\times(0,\pi/2)$ is a homeomorphic embedding. \\
\end{itemize}
Choose any $\zeta_1\in\partial X_1$ and let $U\subset\partial X_1$ be an open neighborhood of $\zeta_1$
(in the cone topology).  Then $U'=q\bigl(U\times\partial X_2\times[0,\pi/4)\bigr)$ is a nonempty open set whose intersection
with $\partial X_1$ is $U$.  By hypothesis, there is a round sphere $K$ which intersects $U'$ at a point,
say $q(\eta_1,\eta_2,t)$.  By the previous corollary, we know that $K=K_1\ast K_2$ where $K_1$ is a round sphere
in $\partial X_1$ and $K_2$ is a round sphere in $\partial X_2$.  Therefore $U\cap K_1$ contains the point $\eta_1$.
\end{proof}

\RalfSubsection{A Splitting Lemma}

In this section we prove a strengthened version of Theorem~\ref{spatziers lemma} from the Introduction. In particular, by Corollary~\ref{density of round spheres} and Lemma~\ref{density in factors}, this theorem applies to irreducible factors of a CAT(0) space with a geometric group action, which is precisely what we need for the proof of Theorem~\ref{main}.

\begin{Thm}
\label{spatziers}
Let $X$ be a proper CAT(0) space, where $\partial X$ has a dense family of round spheres.  Suppose
$\partial X$ contains a proper subset $P$ that is $\pi$-convex and involutive and closed in the
cone topology.  Then $\partial X$ splits as the spherical join $\partial X=P\ast P^\perp$
where $P^\perp$ is the set of points that have Tits distance exactly $\pi/2$ from all points in $P$.
\end{Thm}

Recall that for a subset $P$ of a metric space $X$,
\[
	P^\perp=\bigl\{\;x\in\partial X\;\big|\;d_{Tits}(x,y)=\frac{\pi}{2}\;\textrm{ for all }\;y\in P\;\bigr\}.
\]

To prove Theorem~\ref{spatziers} we first need to establish some lemmas.

\begin{lemma}
Let $P$ be an involutive subset of $\partial X$ and $K$ be a round sphere in $\partial X$.
Define
\[
	P^\perp_K
		=
	\bigl\{\;x\in K\;\big|\;d_{Tits}(x,y)=\frac{\pi}{2}\;\textrm{ for all }\;y\in P\cap K\;\bigr\}.
\]
Then $P\cap K$ is nonempty and $P^\perp_K=P^\perp\cap K$.
\end{lemma}

\begin{proof}
By Lemma~\ref{antipodes are}, every point of $P$ has an antipode in $K$.  Thus $P\cap K$ is nonempty.
It is clear that $P^\perp\cap K\subset P^\perp_K$.
To prove the converse, let $x\in P^\perp_K$ and $y\in P$.
Draw the geodesic $[y,x]$ and extend inside $K$ to get a geodesic
$[y,z]$ of length $\pi$ that passes through $x$. Since $P$ is involutive, $z \in P$. Thus $d(x,z)=\pi/2$ by definition of $P^\perp_K$ and so $d(x,y)=\pi/2$ as well.  
\end{proof}

\begin{lemma}
If $P\subset\partial X$ is an involutive set, then $P^\perp$ is closed in the cone topology.
\end{lemma}

\begin{proof}
Let $\{\zeta_n\}\subset P^\perp$ be a sequence of points converging to a point $\zeta\in\partial X$. 
Then for every $\eta\in P$, $d(\zeta,\eta)\le\pi/2$ since the Tits metric is lower
semicontinuous in the cone topology.  But since $X$ is geodesically complete,
$\zeta$ has an antipode $\zeta'$.  Since $d(\zeta',\eta)\le\pi/2$, we must have
$d(\zeta,\eta)=\pi/2$.
\end{proof}

\begin{proof}[Proof of Theorem~\ref{spatziers}]
Our goal is to prove that every $x\in\partial X$ lies between points $y\in P$ and $z\in P^\perp$.
Let $x\in\partial X$ be given.  By assumption
we may choose round spheres $K_n$ and $x_n\in K_n$ such that $x_n\to x$.
If for large enough $n$, $P\supset K_n$, then $x\in P$ and we are done.  On the other hand, every $P_n=P\cap K_n$ is
nonempty by the lemma above.  So $P_n$ is a proper closed involutive
$\pi$-convex subset of $K_n$, hence a subsphere, hence $K_n=P_n\ast P^\perp_n$ where $P_n^\perp=P^\perp\cap K_n$.
Choose $y_n\in P_n$ and $z_n\in P_n^\perp$ such that $x_n\in[y_n,z_n]\subset K_n$.  Since $\partial X$ is compact,
we may pass to subsequences so that $y_n\to y\in\partial X$ and $z_n\to z\in\partial X$.  Since $P$ and $P^\perp$ are both
closed, $y\in P$ and $z\in P^\perp$.

It remains to verify that $d(y,x)+d(x,z)=\pi/2$.
Since the Tits metric is lower semicontinuous in the cone topology,
\[
	\pi/2=d(y,z) \le d(y,x)+d(x,z) \le \liminf d(y_n,x_n)+d(x_n,z_n) = \pi/2.
\]
\end{proof}

\begin{proof}[Proof of Theorem~\ref{spatziers lemma}]

By Corollary~\ref{density of round spheres} there is a dense family of round spheres in $\partial X$. Hence the claim follows from Theorem~\ref{spatziers}.
\end{proof}

\RalfSubsection{Almost-Flat Triangles}

Recall that the boundary of a flat sector in $X$ is a Tits geodesic in $\partial_{Tits}X$ by \cite[Corollary~II.9.9]{Bridson1999}.
While the converse is not true in general, it is true approximately
as we will see in this section.
We will need two lemmas in Euclidean geometry:
the well-known Alexandrov Lemma and a controlled version.

\begin{lemma}\cite[Lemma~I.2.16]{Bridson1999}
\label{Aleksandrov}
Let
	$\overline x$,
	$\overline y$,
	$\overline z$,
	$\overline w$,
	$\widetilde x$,
	$\widetilde y$,
	$\widetilde z$, and
	$\widetilde w$ be points in $\E^2$ such that 
	
	\begin{enumerate}
	\item $\overline w$ is between $\overline y$ and $\overline z$,
	\item $\widetilde z$ and $\widetilde y$ are on opposite sides
of the line passing through $\widetilde x$ and $\widetilde w$, 

\item $d(\widetilde y,\widetilde x)
	+d(\widetilde x,\widetilde z)
		\ge
d(\widetilde y,\widetilde w)
	+d(\widetilde w,\widetilde z), $
\item $d(\overline x,\overline y) =d(\widetilde x,\widetilde y), 
d(\overline x,\overline z) =d(\widetilde x,\widetilde z), 
d(\overline y,\overline w) =d(\widetilde y,\widetilde w), 
d(\overline w,\overline z)=d(\widetilde w,\widetilde z), $ and
\item $\pi
	\le
		\angle_{\widetilde{w}} (\widetilde{x}, \widetilde{y})+\angle_{\widetilde{w}} (\widetilde{x}, \widetilde{z}).$ \\
Then
\item $\angle_{\overline x}(\overline y,\overline z) \ge
		\angle_{\widetilde x}(\widetilde y,\widetilde w)
			+
		\angle_{\widetilde x}(\widetilde w,\widetilde z), $
\item $\angle_{\overline y}(\overline x,\overline w) \ge
		\angle_{\widetilde y}(\widetilde x,\widetilde w), $
\item $\angle_{\overline z}(\overline x,\overline w) \ge
		\angle_{\widetilde z}(\widetilde x,\widetilde w), $ and
\item $
	d(\overline x,\overline w) \ge d(\widetilde x,\widetilde w).$
\end{enumerate}
\end{lemma}

The next lemma is a modified version of the Alexandrov Lemma, which gives a lower
bound on $d(\widetilde x,\widetilde w)$ under additional hypotheses.

\begin{lemma}[Controlled Alexandrov Lemma]
\label{approximate alex lemma}
Let $0<\theta<\pi$ be fixed.  Given $\epsilon>0$, there is a $\delta>0$
such that whenever $\overline x$, $\overline y$, $\overline z$, $\overline w$,
$\widetilde x$, $\widetilde y$, $\widetilde z$, and $\widetilde w\in\E^2$ satisfy
the conditions of the Alexandrov lemma and in addition:
\begin{itemize}
\item $d(\overline x,\overline y)=d(\overline x,\overline z)=1$,
\item $\overline w$ is the midpoint of the segment $[\overline z,\overline y]$,
\item $\theta-\delta\le \angle_{\widetilde x}(\widetilde y,\widetilde z)
	\le \angle_{\overline x}(\overline y,\overline z)\le\theta$
\end{itemize}
then $d(\overline x,\overline w)-\epsilon\le d(\widetilde x,\widetilde w).$
\end{lemma}

\begin{proof}
Let $0<\theta<\pi$ be given.  Suppose $\overline x$,
$\overline y$, and $\overline z$ satisfy the hypotheses.
Without loss of generality, choose $\widetilde x=\overline x$ and
$\widetilde y=\overline y$.
Let $C_{\overline z}$ be the subarc of the circle of radius 1 centered at $\overline x$ joining
$\overline z$ to $\overline y$ which has length $<\pi$.
Given $\widetilde z$, the circles centered at $\widetilde z$
and $\widetilde y$ of radius $d(\overline z,\overline w)$ must intersect in one or two points
because of (6) in the previous lemma.
Condition (5) guarantees that $\widetilde w$ is the point closer to $\widetilde x$.
This shows that $d(\widetilde w,\overline w)$ is a continuous function of the pair
$(\widetilde z,\overline z)$ whose domain is a compact set (see Figure~\ref{fig:aal}).
Therefore this function is uniformly continuous and since it attains zero whenever
$\widetilde z=\overline z$, the conclusion follows.
\end{proof}

\begin{figure}

\begin{tikzpicture}
	\begin{pgfonlayer}{nodelayer}
		\node [style=point] (x) at (0, 0) {};
		\node [style=point] (y) at (5.66, 0) {};
		\node [style=point] (wb) at (4.83,2) {};
		\node [style=point] (wt) at (4.09,1.49) {};
		\node [style=point] (zb) at (4, 4) {};
		\node [style=point] (zt) at (4.34, 3.64) {};
		\node (6) at (-.5, -0.25) {$\overline{x} = \widetilde{x}$};
		\node  (7) at (6, -0.25) {$\overline{y}= \widetilde{y}$};
		\node (8) at (3.99,1.19) {$\widetilde{w}$};
		\node (9) at (5.13,2.05) {$\overline{w}$};
		\node(10) at (4.5, 4) {$\widetilde{z}$};
			\node (11) at (4.1, 4.3) {$\overline{z}$};
		\node (12) at (2, 0) {};
		\node [above right] (13) at (1.6, .95) {$\le\theta$};
		\node (14) at (4,2.8)  {};
		\node (15) at (5.01,0.51) {};
		\node (wbarrow) at (4.83,2) {};
		\node (17) at (4.9,2.83) {};
			\node (18) at (3.2, 4.3) {$C$};
	\end{pgfonlayer}

	\begin{pgfonlayer}{edgelayer}
		\draw [ultra thick, red] (x) to (y);
		\draw [ultra thick, blue] (zb) to (wb);
		\draw [ultra thick, blue] (wb) to (y);
		\draw [ultra thick, blue] (zt) to (wt);
		\draw [ultra thick, blue] (wt) to (y);
		\draw[ultra thick, red] (zt) to (x);
		\draw[ultra thick, red] (zb) to (x);
		\draw[] (x) to (wt);
		\draw[] (wb) to (x);
		\draw  (12) arc (0:45:2cm);
		\draw[thick, dashed, teal] (17) arc (30:55:5.66cm);
		\draw[arrows=->, thick, teal] (wt) to (wbarrow);
			\end{pgfonlayer}
\end{tikzpicture}
\caption{The Proof of the Controlled Alexandrov Lemma}
\label{fig:aal}
\end{figure}
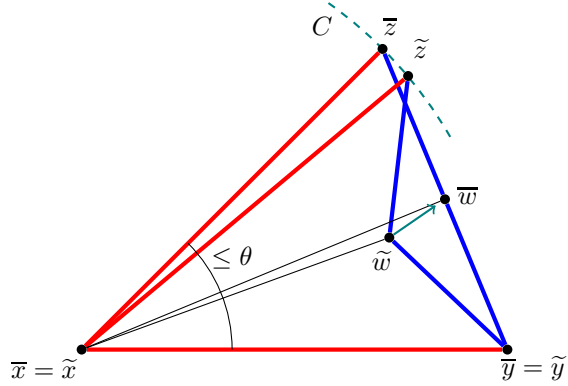

We now return our attention to studying triangles in CAT(0) spaces.
The proof of the next lemma echoes that of the Flat Triangle
Lemma in \cite[Proposition~II.2.9]{Bridson1999}. 

\begin{lemma}[Approximately Flat Triangle Lemma]
\label{approximately flat triangles}
Let $(X,d)$ be a CAT(0) space and let $\theta\in(0,\pi)$ be fixed.
Given $\epsilon>0$, there is a $\delta>0$ such that
for all $t>0$, whenever $x,y,z\in X$ with $d(x,y)=d(x,z)=t$ and

\[
	\theta-\delta\le\angle_x(y,z)\le\overline{\angle_x}(y,z)\le\theta,
\]
then for the midpoint $w$ in the geodesic $[y,z]$,
\[
	d(w,x)\ge d_{\E^2}(\overline w,\overline x)-\epsilon t
\]
where $\triangle(\overline x,\overline y,\overline z)$ is a
comparison triangle in Euclidean space with $\overline w$
the point corresponding to $w$ (i.e. $\overline w$ is the midpoint
of $[\overline y,\overline z]$).
\end{lemma}

\begin{proof}
Choose $\widetilde x,\widetilde y,\widetilde z,\widetilde w$ in $\E^2$ so that
$\triangle(\widetilde x,\widetilde y,\widetilde w)$ and
$\triangle(\widetilde x,\widetilde z,\widetilde w)$ are comparison triangles for
$\triangle(x,y,w)$ and $\triangle(x,z,w)$ respectively, arranged so that $\widetilde z$
and $\widetilde y$ are on opposite sides of the the line passing through $\widetilde x$ and $\widetilde w$.
The reader may check that the hypothesis of Alexandrov Lemma (Lemma~\ref{Aleksandrov}) are satisfied.
Then
\begin{align*}
	\theta-\delta
&\le
	\angle_{x}(y,z) &(\textrm{by hypothesis}) \\
&\le
	\angle_x(y,w)+\angle_x(w,z) &(\textrm{by triangle inequality})\\
&\le
	\overline{\angle}_x(y,w)+\overline{\angle}_x(w,z) &(\textrm{by definition of comparison angle})\\
&=
	\angle_{\widetilde x}(\widetilde y,\widetilde w)
		+
	\angle_{\widetilde x}(\widetilde w,\widetilde z)=\angle_{\widetilde x}(\widetilde y,\widetilde z) &(\textrm{by euclidean geometry}) \\
&
	\le\overline{\angle}_{x}(y,z) &(\textrm{by (6) of Lemma~\ref{Aleksandrov}})\\
&
	\le\theta &(\textrm{by hypothesis}.)
\end{align*}
This verifies the last hypothesis of Lemma~\ref{approximate alex lemma}.
Therefore the claim follows after rescaling.
\end{proof}

\begin{corollary}[Approximately Flat Sectors]
\label{approximately flat sectors}
Let $\zeta,\nu\in\partial X$ such that $\theta=d_{Tits}(\zeta,\nu)<\pi$ and $\epsilon>0$
be given.  Then $p\in X$ may be chosen so that the following statement holds:
if $\alpha$ and $\beta$ are the unit speed geodesics emanating from $p$ going out to $\zeta$ and
$\nu$ respectively and $z_t$ is the midpoint of the geodesic
$[\alpha(t),\beta(t)]$, then $d(p,z_t) \ge t\cos(\theta/2)-\epsilon t$.
\end{corollary}

\begin{proof}
By definition of $\angle_{Tits}(\zeta,\nu)$, $p\in X$ may be chosen so that $\theta-\angle_p(\zeta,\nu)$
is smaller than the $\delta=\delta(\theta,\epsilon)$ provided in Lemma~\ref{approximately flat triangles}.
For any fixed $t$, the comparison triangle $\overline{\triangle}(p,\alpha(t),\beta(t))$ in Euclidean
space is an isosceles triangle with two sides of length $t$ and apex angle of measure
$\theta_t = \overline{\angle_p}(\alpha(t),\beta(t))$.  Since $\theta_t$ is nondecreasing in $t$
(\ref{angle function}), $\theta_t$ lies between $\theta-\delta$ and $\theta$.
Therefore, by Lemma~\ref{approximately flat triangles},
\[
	d(p,z_t) \ge t\cos(\theta_t/2)-\epsilon t \ge t\cos(\theta/2)-\epsilon t.
\]
\end{proof}

\section{Affine Maps}
\RalfSubsection{Properties of Affine Maps Between CAT(0) Spaces}

Let $f\colon X\to Y$ be a continuous affine map between proper CAT(0) spaces.
We first establish Lemma~\ref{prop:surjective} below, which allows us to assume that $f$ is surjective.
To this end, recall that CAT(0) metrics are convex, meaning that the distance between
a pair of points on geodesics is bounded above by a convex combination of the distances
between their endpoints.  This implies that geodesic segments are uniquely determined
by their endpoints.  Moreover, after reparameterizing as constant speed maps over $[0,1]$,
they depend continuously on their endpoints (in the uniform topology on maps).

\begin{lemma}
\label{prop:surjective}
Let $f\colon X\to Y$ be a continuous affine map between proper CAT(0) spaces.
The image $Y'$ of $f$ is a closed, convex subspace of $Y$.
\end{lemma}

\begin{proof}
Any two points $p' \neq q' \in Y'$ are images of points $p \neq q \in X$. Hence $p'$ and $q'$ belong to $f([p,q])$. Since geodesics are unique, it follows that $Y'$ is convex. 

To see that $Y'$ is closed,
choose any sequence $(y_n)\subset Y'$ converging to a point $y\in Y$.
Choose preimages $(x_n)\subset X$.  By passing to a subsequence, we
may assume that $x_n\to x\in X\cup\partial X$.  If $x\in X$, then
$y=f(x)\in Y'$ by continuity and we are done.  Otherwise the sequence
of geodesics $[x_0,x_n]$ converges to a ray $\alpha$.  By
continuity of geodesics in their endpoints, $[y_0,y_n]\to f(\alpha),$ which is either a ray
or a point.  If it is a ray, then $\{y_n\}$ is unbounded, giving
us a contradiction.  So $f(\alpha)$ is a point and $y=y_0\in Y'$.
\end{proof}

\begin{rem}
The continuous affine image of a CAT(0) space need not be CAT(0).
For instance, the identity map $\E^2\to(\R^2,l_1)$ is a continuous affine map.
Similar examples can be obtained by replacing $l_1$ with a norm determined
by a suitable centrally symmetric convex body.  Indeed, this idea gives rise to the
Finsler norms explored in \cite{Lytchak2012}.
\end{rem}

Recall that for a geodesic $\alpha$ in $X$, $\rho(\alpha)$ denotes the constant by which $\alpha$
is rescaled by $f$.

\begin{lemma} 
\label{extendrho}
Let $f\colon X\to Y$ be a continuous affine map between proper CAT(0) spaces
with rescaling function $\rho$.
Then:
\begin{enumerate}
\item $\rho$ determines a function $\partial X\to[0,\infty)$, which we still call $\rho$.
	In other words, $\rho(\alpha)=\rho(\beta)$ if $\alpha$ and $\beta$ are asymptotic geodesic rays in $X$.
\item If $\rho(\alpha)>0$ and $\beta$ is a geodesic ray asymptotic to $\alpha$, then $f(\alpha)$ and $f(\beta)$ are also asymptotic geodesic rays.
\item $\rho$ is a continuous function on $\partial X$ in the cone topology.
\item If $f$ is surjective and there is no $\zeta\in\partial X$ such that $\rho(\zeta)=0$, then
$f$ extends to a homeomorphism $X\cup\partial X\to Y\cup\partial Y$.
\end{enumerate}
\end{lemma}

\begin{proof}
Assume all geodesics in $X$ are parameterized to have unit speed.
Suppose $\alpha$ and $\beta$ are a pair of asymptotic geodesics in $X$.
We will break the proof up into two cases.

{\bf Case 1:} Assume $\rho(\alpha)>0$.  Then the image
of $\alpha$ is also a geodesic ray (with new speed).
Let $\gamma_n$ be the geodesic joining
$\beta(0)$ to $\alpha(n)$.
Then, by the triangle inequality,
\[
	\lim_{n\to\infty}
		\frac{d\bigl(\beta(0),\alpha(n)\bigr)}{d\bigl(\alpha(0),\alpha(n)\bigr)}
			=
	\lim_{n\to\infty}
		\frac{d\bigl(f\beta(0),f\alpha(n)\bigr)}{d\bigl(f\alpha(0),f\alpha(n)\bigr)}
			=
	1.
\]
This implies that $\rho(\gamma_n)\to\rho(\alpha)$.

By convexity of the metric, $\gamma_n\to\beta$ uniformly on
compact sets.  In $Y$, $f(\gamma_n)$ is the geodesic joining
$f(\beta(0))$ to $f(\alpha(n))$ and
$f(\gamma_n)$ converges to the unique geodesic
ray $\widehat{\beta}$ emanating from $f(\beta(0))$ which is
asymptotic to $f(\alpha)$.
On the other hand, $f(\gamma_n)$ converges to $f(\beta)$.
This implies that $f(\beta)=\widehat{\beta}$
and hence $\rho(\alpha)=\rho(\beta)$.  This establishes (1) and (2).

To get (3), observe that whenever $\gamma_n$ is a
sequence of geodesics (either segments or rays) converging to a geodesic
ray $\beta$, then $f(\gamma_n)$ converges to $f(\beta)$ and $\rho(\gamma_n)$
converges to $\rho(\beta)$.  For (4), we observe that since $\rho$ is continuous
on the compact set $\partial X$, it attains upper and lower bounds.  Therefore
$f$ is bi-Lipschitz and $f^{-1}$ is continuous.

{\bf Case 2:} Assume $\rho(\alpha)=0$.
Then the image of $\alpha$ is a single point, and all $\gamma_n$
have the same image -- a finite geodesic segment $\widehat{\gamma}$
emanating from $f\beta(0)$.  Since the lengths of the $\gamma_n$ go to
infinity, $\rho(\gamma_n)$ converges to zero.
Again, $\gamma_n$ converges to $\beta$, and so by continuity of $f$,
$\rho(\beta)=0$.
\end{proof}

Next we reduce to the case where $f$ is injective.

\begin{lemma}
\label{zerofibers}
Let $f\colon X\to Y$ be a continuous affine map between geodesically complete proper CAT(0) spaces.
If $\rho$ is 0 anywhere on $\partial X$, then $X$ splits as a product  $X= X_0 \times X_1$ such that
$\partial (X_0) = \rho ^{-1} (0)$.  More precisely,

\begin{enumerate}
\item for every $x\in X_0$, $f$ is injective on the subspace $\{x\}\times X_1$ and \\
\item for every $y\in X_1$, $f(X_0\times\{y\})$ is a single point. \\
\end{enumerate}
\end{lemma}

\begin{proof} 
Our proof resembles the proof of \cite[Proposition~II.6.23]{Bridson1999}.
Fix $x$ and $y$ in $Y$, and let $Z_x$ and $Z_y$ be their preimages under $f$.
Observe first of all that each $Z_x$ is totally geodesic.  For, given $p,q\in Z_x$,
$\rho([p,q])=0$.  Thus every ray $\alpha$ which extends $[p,q]$ stays inside $Z_x$.

Now define $\phi\colon Z_x\to[0,\infty)$ by letting $\phi(p)$ be the distance from $p\in Z_x$
to the closest point on $Z_y$.  This function is convex.  We will show that $\phi$ is constant
by proving that it is bounded.  Suppose $\phi(p)>\phi(q)$.  Extend $[q,p]$ to get a geodesic
ray $\alpha$.  Since $Z_x$ is totally geodesic, $\alpha$ is contained in $Z_x$.
The fact that $\phi$ is convex and increasing on $\alpha$ implies that $\phi$ is unbounded
here.

Now, let $q'\in Z_y$ be the point
closest to $q$ and $\alpha'$ be the geodesic ray emanating from $q'$ asymptotic to $\alpha$.
By the previous proposition, $\rho(\alpha')=\rho(\alpha)=0$ and therefore $f(\alpha')=f(q')$.
This shows that $\alpha'\subset Z_y$.  But this means that $\phi$ is bounded on $\alpha$,
giving us a contradiction.

Hence, $Z_x$ and $Z_y$ are equidistant totally geodesic subspaces.  As in the proof of
of \cite[Proposition~II.6.23]{Bridson1999}, we may use the Sandwich Lemma
\cite[Exercise~II.2.12(2)]{Bridson1999} to get an isometry $X\to X_0\times X_1$
where $X_0$ is the set $\{Z_x\}_{x\in X}$ and $X_1=Z_{x_1}$ for some $x_1\in X$.
\end{proof}

From this point forward, we will assume that $\rho$ is bounded away from zero, which is equivalent to the assumption that $f$ is injective. 

We will need to know that $\rho$ is constant on antipodes.
If an antipodal pair $\zeta,\eta\in\partial X$ is joined by a geodesic line in $X$,
then this statement is obvious.  This is guaranteed when $d_{Tits}(\zeta,\eta)>\pi$,
for instance, by \cite[Proposition~II.9.21(1)]{Bridson1999}.
However, one can construct CAT(0) 2-complexes in which there is a pair of points $\zeta,\eta$ in
the boundary where $d_{Tits}(\zeta,\eta)=\pi$ but there is no geodesic in the
space joining them \cite[Example~II.9.23(2)]{Bridson1999}.
Therefore a more robust argument is needed.

First, a technical lemma about Euclidean space.

\begin{lemma}
Let $0<m\le M$ be given.  Then there is a continuous function $\sigma:[0,\pi)\to(0,\infty)$
such that $\sigma(0)=0$ and whenever $\triangle xyz$ is an isoceles triangle in Euclidean space,
with $d(x,y)=d(x,z)$ and $\triangle x'y'z'$ is another triangle satisfying
\[
	\frac{d(x,y)}{d(x',y')},
	\frac{d(x,z)}{d(x',z')},
		\textrm{ and }
	\frac{d(y,z)}{d(y',z')}
\]
lying between $m$ and $M$, then
\[
	\angle_{x'}(y',z')\le\sigma(\angle_{x}(y,z)).
\]
\end{lemma}

\begin{proof}
Begin by defining $D(\theta)$ to be the length of the third side of an isoceles triangle with
legs 1 and apex angle of measure $\theta$.  Define
$\sigma(\theta)=\arccos\left(1-\frac{M^2}{2m^2}D(\theta)^2\right)$.
This is continuous in $\theta$
with $\sigma(0)=0$.  Let $\triangle xyz$ and $\triangle x'y'z'$ be as in the statement
of the lemma.  Denote $t=d(x,y)=d(x,z)$ and $\theta=\angle_{x}(y,z)$.  Then $d(y,z)=D(\theta)t$.
Set also $A=d(x',y')$, $B=d(x',z')$, $C=d(y',z')$, and $\theta'=\angle_{x'}(y',z')$.
From the Law of Cosines and since $A^2+B^2\ge 2AB$ for all real numbers $A$ and $B$, we obtain
\begin{align*}
	\cos\theta'
		&=
	\frac{A^2+B^2}{2AB}-\frac{C^2}{2AB} \\
		&\ge
	1-\frac{M^2}{2m^2}D(\theta)^2 \\
		&=
	\cos\bigl(\sigma(\theta)\bigr).
\end{align*}
Since $\cos(\theta)$ is decreasing in $\theta$, we have $\theta'\le\sigma(\theta)$ as desired.
\end{proof}

\begin{lemma}[Small Angles Lemma]
\label{small angles lemma}
Let $X$ and $Y$ be CAT(0) spaces and $f:X\to Y$ be a continuous affine map.
For every $\epsilon>0$, there is a $\delta>0$ such that
whenever $x,y,z\in X$ are distinct and $\angle_{x}(y,z)<\delta$,
then $\angle_{f(x)}(f(y),f(z))<\epsilon$.  Furthermore,
for every $\epsilon>0$ there is a $\delta>0$ such that
whenever $x,y,z\in X$ are distinct and $\angle_{x}(y,z)<\delta$,
then $\bigl|\rho([x,y]) - \rho([x,z])\bigr|<\epsilon$.
\end{lemma}

\begin{proof}

Let $x,y,z\in X$ be distinct, and let $\sigma$ be the function constructed in the
previous lemma where $m$ and $M$ are the minimum and maximum of $\rho$.
Let $\alpha$ and $\beta$ denote the unit speed geodesics $[x,y]$ and $[x,z]$
with $\alpha(0)=\beta(0)=x$ and choose any $0<t<\min\{d_X(x,y),d_X(x,z)\}$.
Recall that $\overline{\angle}_{x}(y,z)$ denotes the angle of a Euclidean
comparison triangle at the vertex corresponding to $x$.  Observe that the ratios
\[
	\frac
		{d\bigl(f\alpha(t),f\beta(t)\bigr)}
		{d\bigl(\alpha(t),\beta(t)\bigr)},\;
	\frac
		{d\bigl(f(x),f\alpha(t)\bigr)}
		{d\bigl(x,\alpha(t)\bigr)},
	\textrm{ and }
	\frac
		{d\bigl(f(x),f\beta(t)\bigr)}
		{d\bigl(x,\beta(t)\bigr)}
\]
all lie between $m$ and $M$.  Therefore
\begin{align*}
	\angle_{f(x)}\bigl(f(y),f(z)\bigr)
&\le
	\overline{\angle}_{f(x)}\bigl(f\alpha(t),f\beta(t)\bigr) \\
&\le
	\sigma\Bigl(\overline{\angle}_{x}\bigl(\alpha(t),\beta(t)\bigr)\Bigr) \\
&\to
	\sigma\Bigl(\angle_{x}\bigl(y,z\bigr)\Bigr)
\end{align*}
as $t\to 0$.  Since $\sigma$ is continuous, this establishes the first part of the lemma.

Now consider the second part of the lemma.  Denote
$A=\rho([x,y])$ and $B=\rho([x,z])$.  Applying the triangle inequality
to the triple $(f\alpha(t)$, $f\beta(t)$, $f(x))$
for small $t>0$, we get
\begin{align*}
	\bigl|At-Bt\bigr|
&\le
	d\bigl(f\alpha(t),f\beta(t)\bigr) \\
&\le
	Md\bigl(\alpha(t),\beta(t)\bigr) \\
&=
	MtD(\overline{\angle}_x(\alpha(t),\beta(t)))
\end{align*}
where $D(\theta)$ is the contiuous function defined in the proof of the previous lemma.
Divide both sides by $t$ and let $t\to 0$ to get
\[
	\bigl|A-B\bigr|
\le
	MD(\angle_x(y,z)).
\]
\end{proof}

\begin{lemma}
\label{angles go to zero}
Assume $\angle_{Tits}(\zeta,\zeta')=\pi$, $x_0\in X$, and $\alpha$
a geodesic ray from $x_0$ to $\zeta$.  Then
\[
	\lim_{t\to\infty}\angle_{\alpha(t)}(x_0,\zeta')=0.
\]
\end{lemma}

\begin{proof}
Let $\beta$ be the geodesic ray emanating from $x_0$ going out to
$\zeta'$ and $\epsilon>0$ be fixed.  By \cite[Proposition~II.9.8(3)]{Bridson1999}
we know that when $s$ and $t$ are large enough,
\[
	\angle_{\alpha(t)}\bigl(x_0,\beta(s)\bigr)
		\le
	\angle_{\alpha(t)}\bigl(x_0,\beta(s)\bigr)
		+
	\angle_{\beta(s)}\bigl(x_0,\alpha(t)\bigr)
		\le
	\frac{\epsilon}{2}.
\]
On the other hand, by continuity of Alexandrov angles (with fixed
basepoint) we can increase $s$ even more (if necessary) to guarantee
\[
	\angle_{\alpha(t)}\bigl(x_0,\beta(s)\bigr)
		\ge
	\angle_{\alpha(t)}(x_0,\zeta')-\frac{\epsilon}{2}.
\]
Put the two together to get that
\[
	\angle_{\alpha(t)}(x_0,\zeta')\le\epsilon.
\]
\end{proof}

\begin{lemma}[Involutive Invariance]
\label{lemma:antipodes}
The rescaling function $\rho$ is constant on pairs of antipodes.
Specifically, whenever $d_{Tits}(\zeta,\zeta')\ge\pi$,
then $\rho(\zeta)=\rho(\zeta')$.
\end{lemma}

\begin{proof}
Let $\alpha$ be a geodesic ray going out to $\zeta$.
For each $t>0$, let $\beta_n$ be the geodesic ray based
at $\alpha(n)$ going out to $\zeta'$.  
By Lemma~\ref{angles go to zero}, $\angle_{\alpha(n)}(\alpha(0),\beta_n(1))\to 0$,
so by Lemma~\ref{small angles lemma}, $\rho(\beta_n)$ converges to $\rho(\alpha)$
as $n\to\infty$.  Thus we establish that $\rho(\zeta')=\rho(\zeta)$.
\end{proof}

\RalfSubsection{Proof of Theorem~\ref{main}}
 Finally we address the problem of how to split affine maps.
Recall that a geodesically complete CAT(0) space $X$ is called \textit{irreducible} if it does not split as a product.
By \cite[Theorem~II.9.24]{Bridson1999},
this is equivalent to saying that the boundary does not split as a spherical join
in the Tits topology (if $X$ is geodesically complete).

Let $X$ and $Y$ be proper geodesically complete CAT(0) spaces such that $X$ admits a geometric group action.
In addition, let $f\colon X\to Y$ be a continuous affine map. By Lemma~\ref{prop:surjective}, we may assume $f$ is surjective. Recall that by Lemma~\ref{zerofibers} we may assume that $f$ is injective, since otherwise, $f$ will be a point map on a factor of $X$.

\begin{lemma}
\label{Linesplitting}
Let $X$ and $Y$ be CAT(0) spaces, where $X$ be irreducible and non-Euclidean. Let $f\colon X\times L \to Y$ be an injective affine map, where $L = \R$. Then $f$ splits. More precisely,
\begin{enumerate}
\item the image of $f$ splits as $Y'\times L'$, where $L'$ is a line and $Y'$ is a convex subset of $Y$,
\item for any $p \in L$, $f(X \times \{p\}) = Y' \times f(p)$, and
\item for any $x \in X$, $f(\{x\} \times L) = f(x) \times L'.$
\end{enumerate} 
\end{lemma}

\begin{proof}  

Note that for all $x \in X$, $f(\{x\}\times L)$ is a geodesic, and for any two such $x_1, x_2 \in X$, these geodesics are parallel. From Lemma~\ref{extendrho} $f(\{x_1\}\times L)$ and $f(\{x_2\}\times L)$ are also parallel. By \cite[Proposition II.2.14(2)]{Bridson1999} the image $f(X\times L)$ splits, where the lines of the new splitting correspond to images of lines $f(\{x\} \times L)$. 

Let $\pi_{L'}$ be the projection from the image $f(X\times L)$ to $L'$. To show (2), we need to show that for any $p \in L$, the composition $\pi_{L'} \circ f$ is constant on $X\times \{p\}$. Note that by \cite[Proposition I.5.3(3)]{Bridson1999}, $\pi_{L'}$ is itself an affine map. Then $\pi_{L'}\circ f$ restricted to $X$ is an affine map from an irreducible non-Euclidean CAT(0) space to $\R$. By \cite[Lemma 4.1]{Lytchak-Schroeder}, since $X$ does not split, this function must be constant.

\end{proof}

 Suppose $X$ admits a splitting, and let $X_1$ and $X_2$ be two factors of the splitting, where $X_2$ is non-Euclidean. Choosing any line $L$ in $X_1$, apply Lemma~\ref{Linesplitting} to find that $f(L)$ is perpendicular to all of the image of $X_2$. If one of the factors $f(X_1)$ splits, then applying the same lemma to $f^{-1}$ tells us that $X_1$ itself must split. Thus the image of an irreducible factor of $X$ is an irreducible factor of $Y$, and hence the image of the Euclidean factor of $X$ must be the Euclidean factor of $Y$. 

It remains to show that, when restricted to a non-Euclidean irreducible factor, $f$ must be a dilation. We will do this by showing that otherwise, $X$ admits a splitting with two nontrivial factors.  Therefore, we will assume for the remainder of this subsection that $f$ is not a dilation.

 We will use Theorem~\ref{spatziers}.
Define $\Max\subset\partial X$ to be the subset on which $\rho$ attains its maximum.
Our goal is to prove that $\partial_{Tits}X$ splits as a spherical join with
$\Max$ as one of the factors.
Since $\rho$ is continuous, this set is closed in the cone topology.  It is involutive by
Lemma~\ref{lemma:antipodes}.  We only need

\begin{lemma}[$\Max$ is $\pi$-Convex]
Assume $\zeta,\eta\in\Max$ such that $d(\zeta,\eta)<\pi$, and let $\nu$ be the midpoint of the geodesic
$[\zeta,\eta]\subset\partial_{Tits}X$.  Then $\nu\in\Max$ as well.
\end{lemma}

\begin{proof}
Since $f$ is bijective, $\rho$ cannot attain zero.
So by rescaling the metric on $Y$ (if necessary), we may assume that the maximum
attained by $\rho$ is 1.  Let $\zeta'$ and $\eta'$ be the images of $\zeta$ and $\eta$
and fix $\epsilon>0$.
Set $\theta=d_{Tits}(\zeta,\eta)/2$, and $\theta'=d_{Tits}(\zeta',\eta')/2$.
By Corollary~\ref{approximately flat sectors} we may choose $p'$
such that the rays $\alpha'$ from $p'$ to $\zeta'$ and $\beta'$ from $p'$ to $\eta'$
satisfy the following: if $z_t'$ is the midpoint of $[\alpha'(t),\beta'(t)]$, then
$d(p',z_t')\ge t\cos(\theta')-t\epsilon$.

Now let $p$ be the preimage of $p'$ and $\alpha$ and $\beta$ be the preimages of $\alpha'$
and $\beta'$.  Note that $\alpha$ and $\beta$ determine the points $\zeta$ and $\eta$ at infinity.
Denote by $z_t$ the midpoint of $[\alpha(t),\beta(t)]$.  By construction, $f(z_t)=z_t'$.
By the Law of Cosines and comparison geometry and \ref{titsangles}, we may choose $t$ large enough
so that
\[
	d(p,z_t)
		\le
	t\cos\Bigl(\overline{\angle}_{p}\bigl(\alpha(t),\beta(t)\bigr)/2\Bigr)
		\le
	t\cos(\theta)+t\epsilon.
\]
%
Thus
\[
	\frac{d(p',z_t')}{d(p,z_t)}
		\ge
	\frac
		{ \cos(\theta')-\epsilon }
		{ \cos(\theta)+\epsilon }.
\]
As shown in the proof of \cite[Lemma~II.9.14]{Bridson1999},
$z_t$ converges to the midpoint $\nu$ of the Tits geodesic $[\zeta,\eta]$ as $t\to\infty$.
Letting $\epsilon\to 0$, we get $\rho(\nu)\ge\cos(\theta')/\cos(\theta)$.
Since we assumed $\rho\le 1$, we have $d(\alpha'(t),\beta'(t))\le d(\alpha(t),\beta(t))$.
So $\theta'\le\theta$, and $\cos\theta'\ge\cos\theta$.  Thus $\rho(\nu)=1$.
\end{proof}


\RalfSubsection{Self-Affine Maps}

Here we consider a self-affine map $f$ of a proper CAT(0) space $X$ admitting a geometric group action.
We will prove Corollary~\ref{thm-selfmaps} that $X$ is flat if $f$ is a strict contraction.
We first need to establish some technical lemmas.

\begin{lemma}
\label{le-wide_sequence}
Let $(X,d)$ be a complete, non-compact metric space.  Then there is a $\lambda>0$ and a sequence
$(x_n)\subset X$ such that for every $m\neq n$, $d(x_m,x_n)\ge\lambda$.
\end{lemma}

\begin{proof}

Since $X$ is not compact, there is a sequence $(y_i)$ with no limit points; thus it cannot have any Cauchy subsequences. As a set, $\{y_i\}$ must contain infinitely many points, since otherwise there would have to be infinitely many $y_i$ all equal to each other and this would give us a convergent subsequence. So we may assume that the $y_i$ are all distinct. Let $x_1 = y_1$. Let $n_1$ be the smallest natural number such that there is some $y_i$ such that $d(x_1, y_i) > 1/n_1$, and let $x_2$ be some such $y_i$. For each subsequent $j \ge 2$, suppose $x_j = y_l$. We will choose $n_j$ and $x_{j+1}$ as follows. Only consider the $y_i$ where $i$ is larger than $l$. If there exist $y_i$ (with $i > l$) such that $d(y_i, x_k) > 1/n_{j-1}$ for all $k \le j$, let $x_{j+1}$ be one such $y_i$, and let $n_{j} = n_{j-1}$. If no such $y_i$ exists, then there must be infinitely many $y_i$, with $i > l$, within the $(1/n_{j-1})$-ball of at least one $x_k$, $k \le j$. Let $z_j$ be one such $x_k$, and, for the rest of the construction, restrict to looking only at the $y_i$, with $i > l$, in this ball. Choose the next $n_{j}$ to be the smallest integer greater than $n_{j-1}$ satisfying $d(z_j, y_i) > 1/n_j$ for some $y_i$ in the $(1/n_{j-1})$-ball of $z_j$, and $x_{j+1}$ to be the $y_i$ of smallest index satisfying this inequality. Continue in this manner, creating an infinite sequence $(x_i)$. If the $n_j$ stabilize at some $N \in \N$, then we have our sequence with $\lambda = 1/N$. If not, then our procedure constructs a sequence $(z_m)$, a subsequence of $(x_i)$. For any $\epsilon > 0$, there is some $j$ such that $n_j > 2/\epsilon$. By construction, all $z_i$  with $i > j$ must be within $\epsilon/2$ of $z_j$, and thus any two are within $\epsilon$ of each other. Thus we have a Cauchy subsequence of $(y_i)$, which is a contradiction.

\end{proof}

Let $R\ge 0$ and $\lambda > 0$.  We say that a subset $\eS\subset\partial X$ is \textit{$(R,\lambda)$-wide} if there is an
$x\in X$ such that for every pair $\zeta,\eta\in\eS$, there is a $y\in B_R(x)$ for which
$\angle_{y}(\zeta,\eta)\ge\lambda$.
We will refer to $x$ as an \textit{$(R,\lambda)$-center} for $\eS$.  There is a bound on the cardinality of wide sets:

\begin{lemma}
Let $X$ be a cocompact proper CAT(0) space, $R \ge 0$, and $\lambda > 0$.  Then there is a bound on the
cardinality of $(R,\lambda)$-wide subsets of $\partial X$.
\end{lemma}

\begin{proof}
Let $\eS\subset\partial X$ be an $(R,\lambda)$-wide subset with $x\in X$ a center.
Let $\alpha,\beta$ be a pair of geodesic
rays emanating from $x$ going out to a pair of points in $\eS$ and set
\[
	K=\frac{2R+1}{\sqrt{2-2\cos \lambda}}.
\]
Let $y\in B_R(x)$ and $\alpha'$ and $\beta'$ be geodesic rays starting at $y$ asymptotic to $\alpha$ and $\beta$
such that $\angle_y(\alpha',\beta')\ge\lambda$.
By \cite[Proposition~II.1.7(5)]{Bridson1999}, $d(\alpha'(K),\beta'(K))$ is bounded below by the length $B$
of the base of an isosceles triangle with legs of length $K$ and apex angle $\lambda$.
Using the Law of Cosines, $B=2R+1$.
By convexity of metric, $\alpha'(K)$ and $\beta'(K)$ are a distance of at most $R$ away from $\alpha(K)$
and $\beta(K)$.  Thus $d(\alpha(K),\beta(K))\ge 1$.
Therefore there is a subset $\widehat{\eS}\subset S_K(x)$ with the same cardinality as $\eS$ such that the
distance between every pair of points is at least $1$.

Suppose now that there is a sequence of $(R,\lambda)$-wide subsets $\eS_n\subset\partial X$ such that
the cardinality of $\eS_n$ is at least $n$ with corresponding center $x_n\in X$.
Using cocompactness, and replacing $\widehat{\eS_n}$ by translates we may assume that $x_n\to x$.
Construct for each $\eS_n$ the corresponding set $\widehat{\eS_n}\subset S_K(x_n)$ as in the previous paragraph.
Choose $y_n^1\in\widehat{\eS_n}$, and let $\widehat{\eS_n^1}$ be the remaining $n-1$ points.
By properness of $X$, we may pass to a subsequence so that $y_n^1\to y^1\in S_K(x)$.
Next choose $y_n^2\in\widehat{\eS_n^1}$ for $n\ge 2$ and let $\widehat{\eS_n^2}$ be the remaining points.
Again, pass to a subsequence so that $y_n^2\to y^2\in S_K(x)$.
Note that $d(y^1,y^2)\ge 1$.  Continuing in this manner, for all $m$ we can find a $y_m\in S_K(x)$
such that $d(y^m,y^n)\ge 1$ for every $n\neq m$.  Thus we have found an infinite discrete subset
of $S_K(x)$, contradicting the assumption that $X$ is proper.
\end{proof}

Our strategy for the proof of Corollary~\ref{thm-selfmaps} will be to show that $\partial_{Tits}X$ is compact, and then apply the following theorem of Bosch\'{e}.

\begin{Thm}\cite[Propositions~3 and 7]{Bosche}
\label{bosche}
Let $X$ be a geodesically complete proper CAT(0) space admitting a geometric group action.  If the Tits
boundary $\partial_{Tits}X$ is compact, then $X$ is flat.
\end{Thm}

\begin{proof}[Proof of Corollary~\ref{thm-selfmaps}]
Suppose $\partial_{Tits}X$ is not compact. By Theorem~\ref{main}, we may pass to a factor of $X$, if necessary, so that $f$ is a dilation. By Lemma~\ref{le-wide_sequence} there is a sequence of points $(x_i)$ in $\partial X$ and a $\lambda > 0$ such that the Tits distance between any two points in the sequence is at least $\lambda$. Let $M$ be the maximal size of $(1,\lambda)$-wide sets of $\partial_{Tits} X$. Let $\eS$ be the first $M+1$ elements of $(x_i)$. Note that we may also require $\lambda < \pi$. In consequence, for any distinct points $\zeta, \nu \in \eS$, we have $\angle_{Tits}(\zeta,\nu) \ge \lambda$.

Form a finite subset $\eS'$ of $X$ by taking for every pair of distinct $\zeta,\nu \in\eS$ a point $y = y(\zeta, \nu) \in X$ such that $\angle_y(\zeta,\nu)\ge\lambda$.

Since $f$ is a dilation, it induces an isometry on $\partial_{Tits}X$.  In particular,
it preserves the distances between the points in $\eS$ while shrinking the diameter of $\eS'$.
Thus for large enough $k$, $f^k(\eS)$ is $(1,\lambda)$-wide, contradicting the assumption that $M$ was the maximal cardinality
of such sets.  Therefore $\partial_{Tits} X$ is compact and we may apply Theorem~\ref{bosche} to get the desired result.

\end{proof}

\begin{proof}[Proof of Corollary~\ref{thm-selfhomeo}]
By the Main Theorem, $X$  splits as a product $X=X_1\times ...\times X_n$
of irreducible factors  $X_i \neq \R$ such that the restrictions $f\mid _{X_i}: X_i \mapsto X$ are dilations.
Since $f$ is a homeomorphism, no rescaling constant can be zero and hence
no $f(X_i)$ is a point.  By uniqueness of the splitting and the Main Theorem, $f$ must interchange the factors.
If $f$ preserves the factors of $X$, then by Corollary~\ref{thm-selfmaps}, it must be an isometry.
This proves (1).  Statements (2) and (3) follow immediately from (1).  For (4), let $j$ be the induced homomorphism from the group of
affine homeomorphisms to the symmetric group of rank $n$.  Then $\ker j$ is a subgroup of isometries by (1).
\end{proof}
%

\begin{rem}
Observe that surjectivity is a necessary assumption since a tree may be dilated by constant $\alpha>1$ to a subtree.
\end{rem}
\bibliography{writeup-arxiv}{}
\bibliographystyle{siam}

\bigskip
\bigskip

\noindent Hanna Bennett\\
hbennett@math.utexas.edu

\bigskip

\noindent Christopher Mooney\\
christopher.mooney.math@gmail.com

\bigskip
\noindent Ralf Spatzier\\
spatzier@umich.edu

\end{document}